\documentclass[a4paper,11pt]{article} 

\usepackage[utf8]{inputenc}
\usepackage[T1]{fontenc}
\usepackage[english]{babel}
\usepackage{indentfirst}
\usepackage{amssymb, amsmath,amsfonts,amsthm,mathrsfs}
\usepackage{latexsym}
\usepackage{wasysym}
\usepackage{enumerate}
\usepackage{xypic}
\usepackage{graphicx}
\usepackage{pgf, tikz}
\usepackage{booktabs}
\usepackage[normalem]{ulem}
\usepackage[makeroom]{cancel}
\usepackage{empheq}

\usepackage{hyperref}
\hypersetup{
	colorlinks,
	citecolor=black,
	filecolor=black,
	linkcolor=black,
	urlcolor=black
	linktoc=all,     
}

 \usepackage{fancyhdr}
 \theoremstyle{plain}
 \newtheorem{thm}{Theorem}[section]
 \newtheorem*{thm*}{Theorem}
 \newtheorem*{thm1.2*}{Theorem 1.2 (for $G=\GL(V)$)}
 \newtheorem{cor}[thm]{Corollary}
 \newtheorem*{cor*}{Corollary}
 \newtheorem{lem}[thm]{Lemma}
 \newtheorem*{lem*}{Lemma}
 \newtheorem{prop}[thm]{Proposition}
 \newtheorem*{prop*}{Proposition}
 \newtheorem{fact}[thm]{Fact}
 \newtheorem*{domanda}{Question}

 \newtheorem*{conj*}{Conjecture}
 \newtheorem*{conjA}{Mann's Conjecture}

 \theoremstyle{definition}
 \newtheorem{defi}[thm]{Definition}

 \newtheorem*{oss}{Remark}
 
 \theoremstyle{remark}

 \newtheorem*{notazione}{Notation}

\newcommand{\Z}{\mathbb{Z}}
\newcommand{\N}{\mathbb{N}}
\newcommand{\F}{\mathbb{F}}

\newcommand{\stab}{\mathrm{stab}}

\newcommand{\cl}{\mathrm{cl}}
\newcommand{\Sym}{\mathrm{Sym}}
\newcommand{\Alt}{\mathrm{Alt}}
\newcommand{\GL}{\mathrm{GL}}

\newcommand{\PGL}{\mathrm{PGL}}

\begin{document}
 
 \begin{titlepage}
 
	\title{A closure operator on the subgroup lattice of $\GL(n,q)$ and $\PGL(n,q)$ in relation to the zeros of the M\"obius function}
	
	\author{Luca Di Gravina \thanks{The author is a member of INdAM research group for Algebraic and Geometric Structures and their Applications (\textsl{GNSAGA}). He is also a member of the research training group \textsl{GRK 2240: Algebro-Geometric Methods in Algebra, Arithmetic and Topology}, funded by DFG. The author thanks both groups for supporting this project.}}
	
	\date{}
  
\end{titlepage}

\maketitle

\begin{abstract}
	Let $\mathbb{F}_q$ be the finite field with $q$ elements and consider the $n$-dimensional $\mathbb{F}_q$-vector space $V=\mathbb{F}_q^n\,$. In this paper we define a closure operator on the subgroup lattice of the group $G = \mathrm{PGL}(V)$. Let $\mu$ denote the M\"obius function of this lattice. The aim is to use this closure operator to characterize subgroups $H$ of $G$ for which $\mu(H,G)\neq 0$. Moreover, we establish a polynomial bound on the number $c(m)$ of closed subgroups $H$ of index $m$ in $G$ for which the lattice of $H$-invariant subspaces of $V$ is isomorphic to a product of chains. This bound depends only on $m$ and not on the choice of $n$ and $q$. It is achieved by considering a similar closure operator for the subgroup lattice of $\mathrm{GL}(V)$ and the same results proven for this group. 
\end{abstract}

\smallskip



\section{Introduction and main results}
  
In this paper, $V$ will denote the vector space $\F_q^n$ of dimension $n$ over the finite field $\F_q$ with $q$ elements. Let $\GL(V)$ and $\PGL(V)$ be, respectively, the general linear group and the projective general linear group of $V$.
We are interested in investigating some properties of the M\"obius function associated to the subgroup lattice of these groups, as described below. To do that, 
we aim to define two suitable closure operators, namely one on the subgroup lattice of $\GL(V)$ and the other one on the subgroup lattice of $\PGL(V)$. These two operators shall be strictly related to each other. Their definitions both rely on the natural permutation action of each group on the subspace lattice of $V$. We note that the action is basically the same for the two groups: the action of $\PGL(V)$ on the subspace lattice of $V$ is in fact just the faithful action induced by the action of $\GL(V)$ on the same lattice.
It will be clear how this yields a correspondence between the closed subgroups in the two subgroup lattices. In particular, closed subgroups in the subgroup lattice of $\PGL(V)$ can be obtained from the corresponding closed subgroups in $\GL(V)$ by 
passing to
the quotient of $\GL(V)$ modulo its centre. In this sense, we can equivalently work with $\GL(V)$ or $\PGL(V)$ (actually, we prefer to work first with $\GL(V)$ when computations are required).
Our results strongly depend on the structure of some sublattices of the subspace lattice of $V$, so that the relation between the subgroup lattice of each group and the subspace lattice of $V$ is central in the paper.   \\

Let $G$ be either $\GL(V)$ or $\PGL(V)$. 
Let $\mathcal{L}(G)$ be the subgroup lattice of $G$ and $\mu:\mathcal{L}(G)\times \mathcal{L}(G)\to \Z$ the M\"obius function on $\mathcal{L}(G)$. The definition of the M\"obius function is given in general for posets in the Preliminaries (see Definition \ref{def_MoebFun}). 
Some combinatorial and topological properties of this function were originally studied for a finite group by Hall in \cite{hall36} (for a historical introduction and some general results, see for instance \cite{hawkes-isaacs-oszaydin89} and \cite{kratzer-thevenaz84}). 

We say that a subgroup $K$ of $G$ is \emph{irreducible} if no non-trivial subspace of $V$ is stabilized by all elements of $K$ under the natural permutation action of $G$ on the subspace lattice of $V$; otherwise the subgroup is said to be \emph{reducible}.
Let $H$ be a subgroup of $G$ and let $\mathit{Irr}(G,H)$ be the set of irreducible subgroups of $G$ that contain $H$. For every $K\in\mathit{Irr}(G,H)$ we denote by $\widehat{\mathcal{I}}(K,H)$ the subposet of $\mathcal{L}(G)$ consisting of $H$, $K$, and all the reducible subgroups of $K$ which contain $H$. Let $\mu_{\widehat{\mathcal{I}}(K,H)}(H,K)$ be
the value in $(H,K)$ of the M\"obius function defined on $\widehat{\mathcal{I}}(K,H)$.

By applying a general result for finite lattices (Lemma \ref{MobFunIdeal} in the Preliminaries), we immediately deduce in Section \ref{section_ClOpLinProjGrps} that 
\begin{equation}\label{eq_muReducibleIdeal2}\tag{$\star$}
\mu(H,G)= \sum_{K\in\mathit{Irr}(G,H)}\mu(K,G)\cdot\mu_{\widehat{\mathcal{I}}(K,H)}(H,K)\,.
\end{equation} 
In order to characterize subgroups $H$ of $G$ with $\mu(H,G)\neq 0$, it is therefore interesting to understand for which pairs of subgroups $H, K$ of $G$ as above we could have $\mu_{\widehat{\mathcal{I}}(K,H)}(H,K)\neq0$.
For this purpose we introduce a closure operator on $\mathcal{L}(G)$, whose definition relies on the action of the group $G$ on the subspace lattice of $V$. The general definition of a closure operator on an abstract lattice will be recalled in the Preliminaries (see Definition \ref{def_ClosureOperator}) and the precise details on the closure operator we choose for $\mathcal{L}(G)$ are available in Section \ref{section_ClOpLinProjGrps}. 
Anyway, for the reader's convenience, we explain here in advance what we mean by ``$H$ closed subgroup in $K$'' in this context, as follows.

Let $H$, $K$ be two subgroups of $G$ so that 
$H\leq K$.
Considering the natural permutation action of $G$ on the subspace lattice of $V$, a subspace of $V$ is \emph{$H$-invariant} if it is stabilized by each element in $H$. 
We say that \emph{$H$ is closed in $K$} if $H$ is exactly the set of all elements in $K$ which stabilize every $H$-invariant subspace of $V$. Otherwise, if there are elements in $K\setminus H$ which stabilize every $H$-invariant subspace of $V$, we say that \emph{$H$ is not closed in $K$}. 

In relation to \eqref{eq_muReducibleIdeal2}, we are interested in the case when $K$ is an irreducible subgroup of $G$ containing $H$, i.e. $K\in\mathit{Irr}(G,H)$. One of the two main results of this paper is the following proposition.  

\begin{prop}\label{NonClosedMu0}
	Let $G$ be one of $\GL(V)$ or $\PGL(V)$, for $V=\F_q^n$.
	Let $H$ be a subgroup of $G$ and $K\in\mathit{Irr}(G,H)$.
	If $H$ is not closed in $K$, then $$\mu_{\widehat{\mathcal{I}}(K,H)}(H,K)=0\,.$$
\end{prop} 

The basic idea here is to exploit the interplay between the subgroup lattice structure of $G$ on
the one hand, and the subspace lattice structure of $V$ on the other hand. The role of the closure operator is essential, as well as the interpretation of the number $\mu_{\widehat{\mathcal{I}}(K,H)}(H,K)$ as the Euler characteristic of a simplicial complex related to  
the lattice of $H$-invariant subspaces of $V$ (see Theorem \ref{thmDVDG1}). \\

Our motivation for studying the values $\mu(H,G)$ is the following conjecture concerning the growth of the M\"obius function on the lattice of open subgroups of positively finitely generated (PFG) profinite groups.

\begin{conjA}[Mann, \cite{mann05}]\label{MannConj} 
	Let $A$ be a PFG profinite group, $\mu$ the M\"obius function on the lattice of open subgroups of $A$, and $b_A(m)$ the number of open subgroups $Y$ in $A$ of index $m$ such that $\mu(Y,A)\neq0$. Then, for every open subgroup $Y\leq_{\,o}A$, the absolute value $|\mu(Y,A)|$ is bounded by a polynomial function in the index $|A:Y|$, and the number $b_A(m)$ grows at most polynomially in $m$. 
\end{conjA}

\begin{notazione} 
	 Mann and other authors denote by $b_m(A)$ the number of subgroups $Y$ of index $m$ in $A$ such that $\mu(Y,A)\neq0$. 
	 We prefer to denote it by $b_A(m)$ to emphasize that this is a function of $m$.
\end{notazione}

Indeed, although the problem is still open in its general setting, it was reduced by Lucchini in \cite{lucchini10} to the study of similar growth conditions for finite almost-simple groups. 
Hence, it is of particular interest for us to address the problem in a way that is independent of the finite vector space $V$ for finite almost-simple groups of the form $\PGL(V)$.

We want to remark that the behaviour of the M\"obius function of a finite group is a question of combinatorial nature of independent interest. So, the fact of considering $\GL(V)$ and finding new related results also for this group is already significant in itself.
Moreover, our results for $\PGL(V)$ are directly deduced from the corresponding ones for $\GL(V)$
through the action of these two groups on the subspace lattice of $V$.
For technical reasons, in Section \ref{section_ProductOfChains} it is more convenient for us to work mainly with $\GL(V)$ and just to observe how we reduce to $\PGL(V)$.  \\

Dealing with Mann's Conjecture, Proposition \ref{NonClosedMu0} together with \eqref{eq_muReducibleIdeal2} suggests  
studying the number of closed subgroups $H$ in $G$ of index $|G:H|=m$, for any $m\in\N$, and computing its growth in $m$. 
The conjecture motivates the interest in a polynomial bound in $m$ for this number (see Corollary \ref{newReduction-b_n(G)}).

\begin{domanda}
	Is the number of closed subgroups $H$ in $G$ of index $m$ bounded by a function $m^\alpha$, for a constant $\alpha$ which is independent of $m$ and $V$?
\end{domanda}  

Actually, we do not find a positive answer for all closed subgroups in $G$ here.
But in Theorem \ref{ClosedSbgrpsProductOfChains} we obtain a positive answer when we consider subgroups $H$ of $G$ with a special structure for the lattice of $H$-invariant subspaces of $V$, as follows. 

Let $\mathcal{S}(V,H)$ be the lattice of $H$-invariant subspaces of $V$. Then the theorem provides a polynomial bound for the number of closed subgroups $H$ in $G$ such that the corresponding $\mathcal{S}(V,H)$ is isomorphic to a product of chains. 
So, let $\mathfrak{X}(G)$ be the subset of $\mathcal{L}(G)$ whose elements are the closed subgroups $H$ in $G$ with $\mathcal{S}(V,H)\simeq\prod_{i}C_i$, for some finite chains $C_i\,$.  

\begin{thm}\label{ClosedSbgrpsProductOfChains}
	Let $G$ be one of $\GL(V)$ or $\PGL(V)$, for $V=\F_q^n$. Let
	$$c(m):=\#\{H\leq G \mid H\in\mathfrak{X}(G), \,|G:H|=m \,\}.$$
	Then there exists an absolute constant $\alpha$, which is independent of $n$ and $q$, such that
	$c(m)\leq m^{\alpha}$ for every $m\in\N$.
\end{thm}

In particular, for $G=\GL(V)$, let $\xi\in G$ be a cyclic endomorphism of $V$ (as characterized in \cite[Theorem 2.1]{neuprae00}) and let $H=\langle \xi \rangle$ be the subgroup of $G$ generated by $\xi$. In this case, we have that $\mathcal{S}(V,H)$ is isomorphic to a product of chains. Then Theorem \ref{ClosedSbgrpsProductOfChains} for $\GL(V)$ can be exploited in the following way.

For a subgroup $H$ of $\GL(V)$, let $\overline{H}\leq \GL(V)$ be the subgroup of all elements in $\GL(V)$ which stabilize every $H$-invariant subspace of $V$. This is nothing but the 
smallest closed subgroup in $\GL(V)$ that contains $H$ and it is called the
\emph{closure} of $H$ in $\GL(V)$.
Clearly, $\mathcal{S}(V,H)=\mathcal{S}(V,\overline{H})$. 
So, let $z(m)$ be the number of closed subgroups in $\GL(V)$ of index $m$  that are the closure of subgroups generated by a cyclic endomorphism. We conclude that $z(m)\leq m^\alpha\,$ for every $m\in\N$. In the final remark of Section \ref{section_ProductOfChains} we point out how this fact could be considered in view of further developments.

\subsection*{Acknowledgement}

I want to thank Francesca Dalla Volta for her careful supervision during the preparation of my Ph.D. thesis, from which this work has been extracted.

\section{Preliminaries}\label{section_Preliminaries}

In general we refer to \cite[Chapter 3]{stanley86} for all definitions and basic facts about partially ordered sets (\emph{posets}) and lattices.
For the reader's convenience, we list below just some  well-known notions of major importance for this paper and we fix some notation.  

We will deal with finite posets and we will mostly reduce to consider finite lattices. 
Let $(L,\leq)$ be a lattice. For any two elements $x,y\in L$, their join in $L$ is denoted by $x\vee y$, their meet by $x\wedge y$.
Chains are a basic example of lattices and we recall the definition. 

\begin{itemize}
	\item A poset $(C,\leq)$ is a \textbf{chain} if $x$ and $y$ are comparable for all $x,y\in C$, i.e. if for any two elements $x,y\in C$ one of either $x\leq y$ or $y\leq x$ holds. \\ Since we assume $|C|<\infty$, we say that the chain $C$ has \textbf{{length}} $\ell\,$ if $\,\ell=|C|-1\,.$ 
\end{itemize}

\noindent In Section \ref{section_ProductOfChains}, we will focus on the case of lattices that are isomorphic to some direct product of chains. 
Let $(Q,\leq)$ be another lattice.

\begin{itemize}
	\item The \textbf{direct product} of $L$ and $Q$ is the lattice $(L\times Q, \leq)$ consisting of the set $L\times Q=\{(x,y)\mid x\in L,\, y\in Q\}$  with $(x,y)\leq (x',y') \text{ in } L\times Q$ if and only if $x\leq x'$ in $L$ and $y\leq y'$ in $Q$. 

	\item $L$ and $Q$ are said to be \textbf{isomorphic} if there exists an order-preserving bijection $\varphi:L\to Q$ whose inverse is order-preserving, so that we have 
	$x\leq y$ in $L$ if and only if $\varphi(x)\leq\varphi(y)$ in $Q$.
\end{itemize} 

\noindent We recall that a lattice isomorphism $\varphi:L\to Q$ preserves the operations of join and meet. 

Let $L$ have a minimum element $\hat{0}$.
\begin{itemize}
	\item Let $u$ be an element of $L$. Then $u$ is said to be \textbf{join-irreducible} in $L$ if $u\neq\hat{0}$, and $u=x\vee y\,$ implies that either $\,u=x\,$ or $\,u=y$. \\
	The subset of all join-irreducible elements in $L$ is denoted by $\mathrm{JI}(L)$. 
\end{itemize} 
\noindent By assuming the lattice $L$ to be finite, we have that every element of $L$ can be written as the join of some join-irreducible elements in $L$. \\

The following kind of subset of $L$, which we usually regard as a subposet of $L$, is important for us (starting from Lemma \ref{MobFunIdeal} below). 

\begin{itemize}
	\item An \textbf{order ideal} of $L$   
	is a subset $I\subseteq L$ such that, for all $x\in I$ and $t\in L$, if $t\leq x$, then  $t\in I$. \\
	In particular, given a subset $A$ of $L$, the set
	$$L_{\leq A} := \{s\in P \mid s\leq a \text{ for some } a\in A\}\subseteq P$$
	is the order ideal in $L$ \textbf{generated by} $A$. 
\end{itemize}

Now we introduce the M\"obius function of a finite poset $P$. A more general, very detailed discussion about the M\"obius function of a poset $P$ and its properties can be found in \cite{stanley86}.   

\begin{defi}\label{def_MoebFun} 
Let $(P,\leq)$ be a finite poset.
The \textbf{M\"obius function} associated with $P$ is the map $\,\mu_{P}:P\times P\rightarrow\Z\,$ defined recursively for $x\leq y\,$ by 
\begin{equation*}\label{EqDef_MobiusFun1}
\mu_{P}(y,y)=1\;  \quad\text{ and }\quad\; \sum\limits_{x\leq t\leq y}{\mu_{P}(t,y)}=0 \;\text{ if }\, x<y\,,
\end{equation*}
and satisfying $\,\mu_{P}(x,y)=0$ unless $x\leq y\,$.	
\end{defi}

M\"obius function and order ideals in finite lattices are related by the following lemma.
The relation will have a central role in Section \ref{section_ClOpLinProjGrps}. 

\begin{lem}\label{MobFunIdeal}
	Let $(L,\leq)$ be a finite lattice with minimum $\hat{0}$ and maximum $\hat{1}$. Let $I\subseteq L$ be an order ideal of $L$. We define:
	\begin{itemize}
		\item $\widehat{I}:=I\cup\{\hat{1}\}$;
		\item $\widehat{I}_{<y}:=\{x\in I\mid x<y\}\cup\{y\}\;$ for every $y\in L\setminus\widehat{I}$.
	\end{itemize} 
	The sets $\widehat{I}$ and $\widehat{I}_{<y}$ are regarded as subposets of $L$. Then
	\begin{equation*}\label{eq_MobFunIdeal1}
	\mu_L(\hat{0},\hat{1})=\mu_{\widehat{I}}(\hat{0},\hat{1}) + \sum_{y\in L\setminus\widehat{I}} \mu_{\widehat{I}_{<y}}(\hat{0},y)\cdot\mu_L(y,\hat{1})\,.
	\end{equation*}
\end{lem}

A direct proof of Lemma \ref{MobFunIdeal} is given in \cite[Lemma 9.3]{godsil18}. 
We will apply Lemma \ref{MobFunIdeal} to express the  M\"obius function of finite linear groups in terms of the ideal of reducible subgroups, as it appears in Theorem \ref{muReducibleIdeal}. In that case, Theorem \ref{thmDVDG1} will provide a suitable expression of  $\mu_{\widehat{I}}(\hat{0},\hat{1})$ and $\mu_{\widehat{I}_{<y}}(\hat{0},y)$ for every $y\in L\setminus\widehat{I}$.

In Section \ref{section_ClOpLinProjGrps}, these results will represent the motivation to introduce the concept of a closure operator on a lattice.
Finally, we recall this notion, which is essential for the main results of this paper. 

\begin{defi}\label{def_ClosureOperator}
	Let $(L,\leq)$ be a lattice.
	A \textbf{closure operator} on $L$ is a map $\;\cl:L\to L$ satisfying the following three conditions: 
	\begin{enumerate}	
		\item $\;\forall\, x\in L\quad x\leq \cl(x)\,$;
		\item $\;\forall\, x,y\in L\quad x\leq y\, \Rightarrow\, \cl(x)\leq \cl(y)\,$;
		\item $\;\forall\, x\in L\quad \cl(\cl(x))=\cl(x)\,$.	
	\end{enumerate} 
	An element $x\in L$  is said to be \textbf{closed} in $L$ if $\cl(x)=x$. For every $y\in L$, the image $\cl(y)$ is called the \textbf{closure} of $y$ in $L$ (through the operator $\cl$).
\end{defi}

Let $L$ be a lattice with closure operator $\cl$. Let $u$ be an element of $L$ and $L_{\leq u}$  the order ideal generated by $\{u\}$. Then $L_{\leq u}$ is not only a subposet, but also a sublattice of $L$. 	The map $\cl_u:L_{\leq u}\to L_{\leq u}$ defined by
$$\forall x\in L_{\leq u}\quad \cl_u(x):=\cl(x)\wedge u$$
is a closure operator on $L_{\leq u}$. 

\begin{defi}\label{def_ClosOp2}
 	Let $x,u\in L$ be so that $x\leq u$. Then $\cl_u(x)$ is called the \textbf{closure of $x$ in $u$}. We say that \textbf{$x$ is closed in $u$} if
 	$\cl_u(x)=x$.
\end{defi} 

Let $G$ be either $\GL(V)$ or $\PGL(V)$ as in the Introduction. Let $H$ and $K$ be two subgroups of $G$, so that $H\leq K$. Thanks to Definition \ref{def_ClosOp2}, we are able to give a precise meaning to the expression ``$H$ is closed in $K$'' in the next section. In particular, we can do that just by defining a single closure operator on the subgroup lattice of $G$. \\

In relation to the study of the M\"obius function of groups, inspiration for the use of closure operators on the subgroup lattice comes from \cite{colombo-lucchini10}. There, closure operators on the subgroup lattices of transitive permutation groups are used to give an answer to Mann's Conjecture for the alternating and symmetric groups $\Alt(n)$ and $\Sym(n)$, with $n\geq 5$.

\section{A closure operator and proof of Proposition \ref{NonClosedMu0}}\label{section_ClOpLinProjGrps}

Let $V=\F_q^n\,$, and let $G=\GL(V)$ or $G=\PGL(V)$. As it is written in the Introduction, we consider the natural permutation action of $G$ on the subspace lattice of $V$. Let $H$ be a subgroup of $G$. 
We denote by $\mathcal{S}(V,H)$ the set of all $H$-invariant subspaces of $V$, that is the set of subspaces $W\leq V$ such that $\,W^h=W$ for every $h\in H$. 
If ordered by inclusion, $\mathcal{S}(V,H)$ 
is a sublattice of the subspace lattice of $V$, where the join operation is given by the sum of subspaces.
Then, by a \textbf{join-irreducible $H$-invariant subspace} we mean a join-irreducible element of $\mathcal{S}(V,H)$.
The set of join-irreducible subspaces in $\mathcal{S}(V,H)$ is denoted by $\mathrm{JI}(\mathcal{S}(V,H))$.

We recall that a subgroup $H$ is \textbf{irreducible} if $\mathcal{S}(V,H)=\{0,V\}$. Otherwise, if there exists some non-trivial $H$-invariant subspace, $H$ is said to be \textbf{reducible}. 
In Aschbacher's classification of maximal subgroups of finite classical groups, the class of maximal reducible subgroups is denoted by $\mathcal{C}_1$: for $G$, it consists of stabilizers of subspaces (see \cite{aschbacher84} and \cite{kleidlieb90} for a detailed description). This is the only Aschbacher class that has some relevance to this work. 
So, we will just write $\mathcal{C}$ when we refer to such a class or, more generally, to some particular families of stabilizers of subspaces of $V$ 
(as in the case of $\mathcal{C}(K,H)$, in Definition \ref{def_objects(K,H)} below).

Let $\mathcal{L}(G)$ be the subgroup lattice of $G$. We define a closure operator on  $\mathcal{L}(G)$ by using the action of $G$ on the subspace lattice of $V$, as follows.

\begin{defi}\label{defClosureOperatorGL/PGL}
	The closure operator on the subgroup lattice $\mathcal{L}(G)$ that we will consider is given by the map
	$\cl:\mathcal{L}(G)\to\mathcal{L}(G)$ such that
	$$\forall H\in\mathcal{L}(G) \quad\quad \cl(H):=\bigcap_{W\in \mathcal{S}(V,H)}\mathrm{stab}_G(W).$$
	We have that $\cl$ is a closure operator in the sense of Definition \ref{def_ClosureOperator}.
	Then $\cl(H)$ is the \textbf{closure of $H$ in $G$}. We say that \textbf{$H$ is closed in $G$} if $H=\cl(H)$. 
\end{defi}

In other words, $H$ is closed in $G$ if and only if there is no element $g\in G\setminus H$ such that $W^g=W$ for every $H$-invariant subspace $W$ of $V$.

\begin{oss}
	Let $\mathrm{JI}(\mathcal{S}(V,H))=\{W_1,\dots,W_r\}$ be the set of all join-irreducible subspaces in $\mathcal{S}(V,H)$. Then
	$$\cl(H)= \mathrm{stab}_G(W_1)\cap\dots\cap\mathrm{stab}_G(W_r).$$
	This is immediate by recalling that every subspace in $\mathcal{S}(V,H)$ is the sum of some join-irreducible subspaces of $\mathcal{S}(V,H)$.
	Hence, $H$ is closed in $G$ if and only if there is no element $g\in G\setminus H$ such that $W^g=W$ for every join-irreducible $H$-invariant subspace $W$ of $V$.
\end{oss}

Let $H$ and $K$ be two subgroups of $G$ so that $H\leq K$. As seen for Definition \ref{def_ClosOp2}, we have that the closure operator $\cl$ on $\mathcal{L}(G)$ induces a closure operator $\cl_K$ on the subgroup lattice of $K$. Thus, we can define the closure of $H$ in $K$ in the following way.   

\begin{defi}\label{defClosure_HinK} 
	Let $H$, $K$ be subgroups of $G$ so that $H\leq K$.
	The \textbf{closure of $H$ in $K$} is the subgroup
	\begin{equation*}
	\cl_K(H):=\cl(H)\cap K= \mathrm{stab}_K(W_1)\cap\dots\cap\mathrm{stab}_K(W_r)
	\end{equation*}
	where $\{W_1,\dots,W_r\}$ is the set of join-irreducible subspaces in $\mathcal{S}(V,H)$.
	So, \textbf{$H$ is closed in $K$} if $H=\cl_K(H)$. 
\end{defi}

Definition \ref{defClosure_HinK} will be applied particularly when $K$ is an irreducible subgroup of $G$. In this case, we also consider the following sets. 
 
\begin{defi}\label{def_objects(K,H)}
	Let $H$, $K$ be subgroups of $G$ so that $H\leq K$.
	Let $K$ be irreducible.  
	We define:
	\begin{itemize}
		\item $\mathcal{C}(K,H):=\{\,\mathrm{stab}_K(W)\mid 0< W< V\,,\,H\subseteq\mathrm{stab}_K(W) \}$, namely the set of stabilizers in $K$ of proper $H$-invariant subspaces of $V$;
		
		\item $\mathcal{I}(K,H):=\{ L\leq K\mid H\leq L\leq M \text{ for some } M\in \mathcal{C}(K,H) \}$, namely the order ideal generated by $\mathcal{C}(K,H)$ in the subgroup lattice of $K$;
		
		\item $\widehat{\mathcal{I}}(K,H):= \mathcal{I}(K,H)\cup\{H,K\}$.
	\end{itemize}
		The subgroups $H$ and $K$ are respectively the minimum and the maximum of $\widehat{\mathcal{I}}(K,H)$, if it is regarded as a subposet of $\mathcal{L}(G)$. 
		
		Since the group $G$ itself is irreducible, we have particularly so defined $\mathcal{C}(G,H)$, $\mathcal{I}(G,H)$, and $\widehat{\mathcal{I}}(G,H)$.
\end{defi}

Now, for a subgroup $H$ of $G$ let
$$\mathcal{L}(G)_{\geq H}:=\{F\in\mathcal{L}(G)\,\mid\,H\leq F\}$$
be the sublattice of $\mathcal{L}(G)$ consisting of all subgroups of $G$ that contain $H$. So, $\mathcal{L}(G)_{\geq H}$ is a finite lattice with minimum $H$ and maximum $G$.
In order to apply Lemma \ref{MobFunIdeal}, we consider $\mathcal{I}(G,H)$ as an ideal of $\mathcal{L}(G)_{\geq H}$. It is the ideal of reducible subgroups of $G$ that contain $H$.

For every $K\in \mathcal{L}(G)_{\geq H} \setminus \widehat{\mathcal{I}}(G,H)$, the set 
$$\widehat{\mathcal{I}}(G,H)_{<K}:=\{X\in\mathcal{I}(G,H)\,\mid\, X<K\}\cup\{K\}\,$$
coincides with $\widehat{\mathcal{I}}(K,H)$.
\begin{notazione}
	For $H\leq G$, we denote: 
	\begin{itemize}
		\item by $\mu$ the M\"obius function of $\mathcal{L}(G)_{\geq H}$;
		\item by $\mathit{Irr}(G,H)$ the set of irreducible subgroups of $G$ that contain $H$;
		\item by $\mu_{\widehat{\mathcal{I}}(K,H)}$ the M\"obius function of $\widehat{\mathcal{I}}(K,H)$, for each $K\in\mathit{Irr}(G,H)$.
	\end{itemize}
\end{notazione}
Hence, 
by Lemma \ref{MobFunIdeal} we have the following theorem, which will be used to prove Proposition \ref{prop_NonClosedMu0} below.

\begin{thm}\label{muReducibleIdeal} 
	Let $H$ be a subgroup of $G$. Then
	\begin{align}
	\mu(H,G)  & =\mu_{\widehat{\mathcal{I}}(G,H)}(H,G) + \sum_{\substack{H< K<G \\ K\notin\mathcal{I}(G,H)}}\mu(K,G)\cdot\mu_{\widehat{\mathcal{I}}(K,H)}(H,K)\, \notag \\ 
	& = \sum_{K\in\mathit{Irr}(G,H)}\mu(K,G)\cdot\mu_{\widehat{\mathcal{I}}(K,H)}(H,K)\, \label{eq_muReducibleIdeal3}
	\end{align}
\end{thm}

Equation \eqref{eq_muReducibleIdeal3} in Theorem \ref{muReducibleIdeal} follows from the fact that $\mu(G,G)=1$ by definition of the M\"obius function. So, we have obtained the relation \eqref{eq_muReducibleIdeal2} of the Introduction.  

Let $K\in\mathit{Irr}(G,H)$.
The number $\mu_{\widehat{\mathcal{I}}(K,H)}(H,K)$ in \eqref{eq_muReducibleIdeal3} has been studied in connection with some topological properties of the M\"obius function. In particular, such a number can be interpreted as the Euler characteristic of a simplicial complex related to the lattice of $H$-invariant subspaces of $V$ (see \cite{dallavolta-digravina22}). 
The following result comes from this interpretation and is used directly in the proof of Proposition \ref{NonClosedMu0} below.

For a subgroup $H$ of $G$ and $K\in\mathit{Irr}(G,H)$ as above, let $\mathcal{S}(V,H)^*$ denote the set $\mathcal{S}(V,H)\setminus\{0,V\}$ of non-trivial $H$-invariant subspaces of $V$, and let	
$$\Psi(K,H):=\{E\subseteq\mathcal{S}(V,H)^*\,\mid\,\bigcap_{W\in E}\mathrm{stab}_K(W)\neq H\}\,.$$
We note that the empty set $\emptyset$ is trivially an element of $\Psi(K,H)$.

\begin{thm}[Theorem 4.5 in \cite{dallavolta-digravina22}]\label{thmDVDG1}
	Let $H$ and $K$ be subgroups of $G$ such that $K$ is irreducible and $H\leq K$.
	Then 
	\begin{equation}\label{eq_thmDVDG1}
	-\mu_{\widehat{\mathcal{I}}(K,H)}(H,K)=\sum_{E\in\Psi(K,H)}(-1)^{|E|}\,.
	\end{equation}
\end{thm}

\begin{oss} 
	Theorem \ref{thmDVDG1} and its proof are given in \cite{dallavolta-digravina22} for $G=\GL(V)$. We observe here that the result holds true also for $G=\PGL(V)$. In fact,
	let $\pi:\GL(V)\to\PGL(V)$ be the canonical projection onto the quotient and let $H_1,\,H_2$ be any two subgroups of $\PGL(V)$. Then, $H_1\leq H_2$ if and only if $\pi^{-1}(H_1)\leq\pi^{-1}(H_2)$. Now, let $H$ and $K$ be subgroups of $G=\PGL(V)$ as in the statement of Theorem \ref{thmDVDG1}. This implies that $\pi^{-1}(K)$ is an irreducible subgroup of $\GL(V)$. So, we have that $\widehat{\mathcal{I}}(K,H)$ and $\widehat{\mathcal{I}}(\pi^{-1}(K),\pi^{-1}(H))$ are isomorphic as lattices. Therefore 
	$$\mu_{\widehat{\mathcal{I}}(K,H)}(H,K)=\mu_{\widehat{\mathcal{I}}(\pi^{-1}(K),\,\pi^{-1}(H))}(\pi^{-1}(H),\pi^{-1}(K)).$$
	Moreover, $\mathcal{S}(V,H)=\mathcal{S}(V,\pi^{-1}(H))$, and hence $E\in\Psi(K,H)$ if and only if $E\in\Psi(\pi^{-1}(K),\pi^{-1}(H))$. Thus, equation \eqref{eq_thmDVDG1} in the theorem follows from the correspondent equation for the subgroups $\pi^{-1}(H)$ and $\pi^{-1}(K)$ of $\GL(V)$.
\end{oss}

So, we can now prove Proposition \ref{NonClosedMu0}. \\

\noindent \textbf{Proof of Proposition \ref{NonClosedMu0}.}
	Let $H$ and $K$ be subgroups of $G$ such that $K$ is irreducible and $H\leq K$.
	We assume that $H$ is not closed in $K$.
	This means that there exists an element $g\in K\setminus H$ such that $W^g=W$ for every join-irreducible $H$-invariant subspace $W$ of $V$.
	Let $W_1,\dots,W_r$ be the join-irreducible elements of $\mathcal{S}(V,H)$.  
	Then $$H\subsetneqq \bigcap_{i=1}^r\mathrm{stab}_K(W_i)\subseteq \bigcap_{W\in E}\mathrm{stab}_K(W)\,$$
	for every non-empty subset $E$ of $\mathcal{S}(V,H)^*\,$. 
	By definition of $\Psi(K,H)$, this implies that $E\in\Psi(K,H)$ for every $E\subseteq\mathcal{S}(V,H)^*\,$ (including $E=\emptyset$). 
	Hence, $\Psi(K,H)$ is the power set of $\mathcal{S}(V,H)^*$ under the assumption that $H$ is not closed in $K$. So, we have that
	$$0=\sum_{E\in\Psi(K,H)}(-1)^{|E|}=-\mu_{\widehat{\mathcal{I}}(K,H)}(H,K)\,,$$
	where the second equality comes from Theorem \ref{thmDVDG1}. This concludes the proof. \qed \\

\subsection*{Some consequences}

Let $\mu$ be the M\"obius function on the subgroup lattice of $G$ and let $H$ be a subgroup of $G$. By Theorem \ref{muReducibleIdeal}, the value $\mu(H,G)$ is characterized through the numbers $\mu(K,G)$ and $\mu_{\widehat{\mathcal{I}}(K,H)}(H,K)$ related to each irreducible subgroup $K\in\mathit{Irr}(G,H)$. For each of these subgroups $K$, information on the number $\mu_{\widehat{\mathcal{I}}(K,H)}(H,K)$ can be obtained by looking at the subspace lattice of $V$ (Theorem \ref{thmDVDG1}). This interplay between the subgroup lattice of $G$ and the subspace lattice of $V$ can be exploited to further characterize subgroups $H$ of $G$ with $\mu(H,G)\neq 0$, as in the following Proposition \ref{prop_NonClosedMu0}. It is basically an application of Proposition \ref{NonClosedMu0}, and therefore an application of the chosen closure operator on the subgroup lattice of $G$. 

\begin{prop}\label{prop_NonClosedMu0} 
	Let $H$ be a subgroup of $G$. If $\mu(H,G)\neq0$, then there exist a subgroup $K\in\mathit{Irr}(G,H)$, with $\mu(K,G)\neq0$, and a closed subgroup $Y$ in $G$ such that $H=K\cap Y$.
\end{prop}
\begin{proof}
	Let $H\leq G$ such that $\mu(H,G)\neq0$. By Theorem \ref{muReducibleIdeal} we know that 
	$$\mu(H,G)=\sum_{K\in\mathit{Irr}(G,H)}\mu(K,G)\cdot\mu_{\widehat{\mathcal{I}}(K,H)}(H,K)\,.$$
	Thus,
	$\mu(H,G)\neq0$ implies that there exists a subgroup $K\in\mathit{Irr}(G,H)$ such that
	$$\mu(K,G)\cdot\mu_{\widehat{\mathcal{I}}(K,H)}(H,K)\neq0.$$
	Then, for such a subgroup $K$ we have both $\mu(K,G)\neq0$ and
	\begin{equation}\label{diseqProp3.6}
		\mu_{\widehat{\mathcal{I}}(K,H)}(H,K)\neq0.
	\end{equation} 
	By Proposition \ref{NonClosedMu0}, we conclude from \eqref{diseqProp3.6} that $H$ is closed in $K$. 
	Then there exists a closed subgroup $Y$ in $G$ such that $H=K\cap Y$. Indeed
	\begin{align*}
	H=\cl_K(H)=\bigcap_{W\in \mathcal{S}(V,H)}\mathrm{stab}_K(W)= K\cap\bigcap_{W\in \mathcal{S}(V,H)}\mathrm{stab}_G(W)=K\cap\cl(H)\,.
	\end{align*}
\end{proof}

The following Corollary \ref{newReduction-b_n(G)} is an immediate consequence of Proposition \ref{prop_NonClosedMu0}. It concerns the number $b_G(m)$ of subgroups $H$ in $G$ of index $|G:H|=m$ such that $\mu(H,G)\neq0$. We recall that the growth of $b_G(m)$ in the index $m$ is of interest for Mann's Conjecture about the M\"obius function of finitely generated profinite groups. 

\begin{cor}\label{newReduction-b_n(G)}
	Let $G$ be one of $\GL(V)$ or $\PGL(V)$, for $V=\F_q^n$.
	Let
	$$b_G(m):=\#\{H\leq G\,\mid\,|G:H|=m \;\text{ and }\; \mu(H,G)\neq0\}.$$ 
	Then there exists an absolute constant $\beta$, which is independent of $n$ and $q$, such that 
	$b_G(m)\leq m^\beta$ for every $m\in\N$,
	if the two following conditions are satisfied:
	\begin{itemize}
		\item[(1)] the number
		$$\#\{K\in\mathit{Irr}(G,H)\,\mid\,\mu(K,G)\neq0\}$$ 
		is bounded by $|G:H|^{\widetilde{\beta}_1}$ for an absolute constant $\widetilde{\beta}_1$ independent of $n$, $q$, and $H\leq G$;
		\item[(2)] the number of closed subgroups in $G$ of index $m$
		is bounded by $m^{\widetilde{\beta}_2}$, for an absolute constant $\widetilde{\beta}_2$ independent of $n$ and $q$. 
	\end{itemize} 
\end{cor}

So, Corollary \ref{newReduction-b_n(G)} says that the investigation on the polynomial growth of $b_G(m)$ could be restricted to the irreducible subgroups $K$ of $G$ with $\mu(K,G)\neq 0$. 
This is possible if we are able to verify the polynomial growth in $m\in\N$ of the number of closed subgroups in $G$ of index $m$.

The next section, with the proof of Theorem \ref{ClosedSbgrpsProductOfChains}, points in this direction. The aim is to exploit some particular structures of sublattices (products of chains) in the subspace lattice of $V$ in order to get information on the number of related closed subgroups in $G$. The result obtained is connected with condition \textit{(2)} of Corollary \ref{newReduction-b_n(G)}. We also suggest extending this approach to a more general count of closed subgroups in $G$.

\section{Proof of Theorem \ref{ClosedSbgrpsProductOfChains} and cyclic endomorphisms}\label{section_ProductOfChains}

In this section we work with $G=\GL(V)$, for $V=\F_q^n\,$.
Let $c(m)$ be the number of closed subgroups $H$ in $G$ of index $m$, for which the lattice $\mathcal{S}(V,H)$ is isomorphic to a product of chains.
As stated in Theorem \ref{ClosedSbgrpsProductOfChains}, we prove that there is a polynomial bound on $c(m)$, i.e. $c(m)\leq m^\alpha$ for an absolute constant $\alpha$ independent of $n$ and $q$.

\begin{oss}
	The choice of restricting to $\GL(V)$ is only technical. The same result can be deduced for $\PGL(V)$, as follows. By Definition \ref{defClosureOperatorGL/PGL}, the centre $Z(\GL(V))$ is contained in every subgroup that is closed in $\GL(V)$. So, it is immediate to obtain Theorem \ref{ClosedSbgrpsProductOfChains} for  $\PGL(V)=\GL(V)/Z(\GL(V))$: 
	let $\pi:\GL(V)\to\PGL(V)$ be the canonical projection and let $Y$ be a subgroup of $\PGL(V)$. Then $Y$ is closed in $\PGL(V)$ if and only if $\pi^{-1}(Y)$ is closed in $\GL(V)$, and moreover
	$|\PGL(V):Y|=|\GL(V):\pi^{-1}(Y)|$.
\end{oss}

Before starting with the proof of Theorem \ref{ClosedSbgrpsProductOfChains} for $G$, we note that there is a typical example of a closed subgroup $H$ in $G$ for which the lattice $\mathcal{S}(V,H)$ is isomorphic to a product of chains.
Such an example is given by the closure of the subgroup generated by a cyclic endomorphism of $V$ in $G$.
We recall this in Proposition \ref{prop_CyclicMatrix-ProductOfChains}, observing that $\mathcal{S}(V,H)=\mathcal{S}(V,\cl(H))$ for every $H\leq G$, by Definition \ref{defClosureOperatorGL/PGL} of the operator $\cl$ on the subgroup lattice of $G$. 

For a reference about cyclic endomorphisms of $V$ in $\GL(V)$ and their possible characterization, see \cite{neuprae00} (actually, since they are regarded as elements of the matrix group $\GL(n,q)$, they are called \emph{cyclic matrices} in \cite{neuprae00}). 
We only recall what we need here, also in view of a final remark at the end of our paper. 
For our purposes, by a cyclic endomorphism of $V$ in $G$ we mean the following.

\begin{defi}\label{CyclicMatrix_def1}
	An element $\xi\in G$ is said to be a \textbf{cyclic endomorphism of $V$} if the characteristic polynomial of $\xi$ is equal to its minimal polynomial.
\end{defi}

For every element $g\in G$ and the correspondent generated subgroup $\langle g\rangle$, clearly $\mathcal{S}(V,\langle g\rangle)$ coincides with the set of subspaces that are $g$-invariant. Thus, we will simply write $\mathcal{S}(V,g)$ to denote it. 

\begin{fact}\label{Corollary_IsoProdOfChains}
	Let $\xi\in G$
	be a cyclic endomorphism of $V$ with minimal polynomial $m_\xi(t)\in\F_q[t]$. We denote by $\mathcal{D}(m_\xi)$ the lattice of all monic divisors of $m_\xi(t)$ in $\F_q[t]$. Then there is an isomorphism of lattices
	$$\mathcal{D}(m_\xi)\simeq \mathcal{S}(V,\xi)$$
	given by the map $\,f(t) \mapsto \ker(f(\xi))$.
\end{fact}
  
For a proof of Fact \ref{Corollary_IsoProdOfChains}, see for example \cite{brifil67}. 
Now, the structure of the lattice $\mathcal{D}(m_\xi)$ is well known: it depends only on the prime factorization over $\F_q[t]$ of the minimal polynomial $m_\xi(t)$ of $\xi$. 
Let 
$$m_\xi(t)=f_1(t)^{\omega_1}\cdot\ldots\cdot f_r(t)^{\omega_r}$$
where $f_1(t),\dots,f_r(t)\in\F_q[t]$ are monic and irreducible, $\omega_1,\dots,\omega_r\in\N$. Then  
$$\mathcal{D}(m_\xi)\simeq \prod_{i=1}^r C(\omega_i) = C(\omega_1)\times\dots\times C(\omega_r)$$
where $\prod_{i=1}^r C(\omega_i)$ denotes the direct product of $r$ chains $C(\omega_1),\dots,C(\omega_r)$ of length, respectively, $\omega_1,\dots,\omega_r$.
So, for a cyclic endomorphism $\xi\in G$, we have the following immediate description of the lattice of $\xi$-invariant subspaces of $V$.

\begin{prop}\label{prop_CyclicMatrix-ProductOfChains}
	If $\xi\in G$
	is a cyclic endomorphism of $V$, then the lattice $\mathcal{S}(V,\xi)$ of $\xi$-invariant subspaces of $V$ is isomorphic to a product of chains.
\end{prop}

Now we can focus on the proof of
Theorem \ref{ClosedSbgrpsProductOfChains}.
It is achieved by studying separately two preliminary cases.
\begin{description}
	\item[Case I.]  We consider the closed subgroups $H$ in $G$ for which the lattice $\mathcal{S}(V,H)$ of  $H$-invariant subspaces of $V$ is boolean (Proposition \ref{ClosedSbgrpsBoolean}).
	\item[Case II.] We consider the closed subgroups $H$ in $G$ for which $\mathcal{S}(V,H)$ is a flag in $V$ (Proposition \ref{ClosedSbgrpsFlag}).
\end{description}
These two conditions can be finally combined to complete the proof of Theorem \ref{ClosedSbgrpsProductOfChains}. We introduce here some useful notation.

\begin{notazione}
	For every $m\in \N$, we will consider the following three sets of closed subgroups in $G$.
	\begin{itemize}
		\item $\mathfrak{B}_m\,$ 
		is the set of closed subgroups $H$ in $G$ such that $|G:H|=m$ and the lattice $\mathcal{S}(V,H)$ is boolean.
		\item $\mathfrak{F}_m\,$
		is the set of closed subgroups $H$ in $G$ such that $|G:H|=m$ and the lattice $\mathcal{S}(V,H)$ is a flag.
		\item $\mathfrak{X}_m\,$ 
		is the set of closed subgroups $H$ in $G$ such that $|G:H|=m$ and the lattice $\mathcal{S}(V,H)$ is isomorphic to a product of chains.
	\end{itemize}
	With this notation, proving Theorem \ref{ClosedSbgrpsProductOfChains} is equivalent to showing that there exists a constant $\alpha$, independent of $n$ and $q$, such that
	$|\mathfrak{X}_m|\leq m^\alpha$ for all $m\in\N$.
\end{notazione}

\subsection*{Proof in Case I}

A description and characterization of boolean lattices can be found in  \cite[Chapter 3]{stanley86}. For our purposes, we only observe that $\mathcal{S}(V,H)$ is boolean if and only if $V=\bigoplus_{i=1}^r W_i\,$ for $\{W_1,\dots,W_r\}=\mathrm{JI}(\mathcal{S}(V,H))$. This is equivalent to saying that $\mathcal{S}(V,H)$ is isomorphic to a product of chains of length 1.

Before proving Proposition \ref{ClosedSbgrpsBoolean}, 
we recall that, for integers $0\leq x\leq n$, the \textbf{$q$-binomial coefficient} is
\begin{equation}\label{q-binomial}
\binom{n}{x}_q=\frac{[n]_q!}{[x]_q!\,[n-x]_q!}
\end{equation}
and represents the number of $x$-dimensional subspaces of $\F_q^n\,$. For the definitions of such $q$-analogues of binomial coefficient and factorial, see for example \cite[Chapter 1]{stanley86}. For what follows, it is useful to write
\begin{equation}\label{CardGLnq}
	|\GL(n,q)|= [n]_q!(q-1)^n\,q^{\binom{n}{2}}\,.
\end{equation}

\begin{oss}
	Let $W_1$ and $W_2$ be two vector subspaces of $V=\F_q^n$ such that $V=W_1\oplus W_2\,$. The dimensions of $W_1$ and $W_2$ are denoted, respectively, by $x_1$ and $x_2$, so that $x_1+x_2=n$.
	Then
	\begin{align}\label{eq_IndexBool}
	\frac{|\GL(V)|}{|\GL(W_1)|\cdot|\GL(W_2)|} &= \frac{[n]_q!(q-1)^n\,q^{\binom{n}{2}}}{[x_1]_q!(q-1)^{x_1}\,q^{\binom{x_1}{2}}[x_2]_q!(q-1)^{x_2}\,q^{\binom{x_2}{2}}} \nonumber\\ \nonumber \\ 
	&= \binom{n}{x_1}_q  q^{\binom{n}{2}-\binom{x_1}{2}-\binom{x_2}{2}} \;=\; \binom{n}{x_1}_q q^{x_1x_2}\,. 
	\end{align}
\end{oss}

\begin{prop}\label{ClosedSbgrpsBoolean}
	Let $V=\F_q^n$ and $G=\mathrm{GL}(V)$. 
	Then there exists an absolute constant $\alpha_1$, which is independent of $n$ and $q$, such that $|\mathfrak{B}_m|\leq m^{\alpha_1}\,$ for every $m\in\N$.
\end{prop}
\begin{proof}
	Let $m\in\N$ be fixed and consider $H\in\mathfrak{B}_m$.
	Since $H$ is closed in $G$,
	$$H=\bigcap_{i=1}^r\mathrm{stab}_G(W_i)$$
	where $\{W_1,\dots,W_r\}=\mathrm{JI}(\mathcal{S}(V,H))$ is the set of  join-irreducible subspaces in $\mathcal{S}(V,H)$.
	Since $\mathcal{S}(V,H)$ is boolean, we have
	$V=\bigoplus_{i=1}^r W_i\,$,
	and consequently
	$$H=\bigcap_{i=1}^r\mathrm{stab}_G(W_i)\simeq\mathrm{GL}(W_{1})\oplus\dots\oplus\mathrm{GL}(W_{r}).$$
	So,
	$|H|=|\mathrm{GL}(W_{1})|\cdot\ldots\cdot|\mathrm{GL}(W_{r})|.$ 
	Let $x_i:=\dim W_i\,$ for all $i=1,\dots,r$. 
	We can inductively use \eqref{eq_IndexBool} in the above remark to obtain that
	\begin{align*}
	m= |G:H| & = \frac{|\mathrm{GL}(V)|}{|\mathrm{GL}(W_{1})|\cdot\ldots\cdot|\mathrm{GL}(W_{r})|} \\ 
	\\
	& = \binom{n}{x_1}_q \cdot \binom{n-x_1}{x_2}_q \cdot \ldots \cdot \binom{n-x_1-\ldots-x_{r-2}}{x_{r-1}}_q \cdot q^\epsilon 
	\end{align*}
	where the exponent 
	$$\epsilon:=\epsilon(n,x_1,\dots,x_{r-1})= x_1(n-x_1)+x_2(n-x_1-x_2)+\ldots+x_{r-1}x_r$$ 
	depends only on $n,x_1,\dots,x_{r-1}$. 
	Let 
	$$v_1=\binom{n}{x_1}_q , \quad v_i=\binom{n-x_1-\ldots-x_{i-1}}{x_i}_q\; \text{ for } i=2,\dots,r-1, \quad v_r=q^{\epsilon}$$
	so that $m=v_1\cdot\ldots\cdot v_r$.
	As in the proof of \cite[Lemma 2.3]{colombo-lucchini10}, by \cite{kalmar31}, the number of such ordered factorizations of $m$ is at most $m^2$. If we fix the factorization $m=v_1\cdot\ldots\cdot v_r$, then there are at most two possible values for $x_1$ to get $v_1$. If $x_1$ is fixed, then there are at most two possible values for $x_2$ to get $v_2$, and so on up to $x_{r-1}$, to get $v_{r-1}$. The remaining $v_r$ in the factorization must be equal to $q^\epsilon$, and $x_r$ is determined by the previous $x_i$, with $i=1,\dots,r-1$. So, for every fixed ordered factorization, we have at most $2^{r-1}$ possibilities, and $2^{r-1}\leq m$. Then there are at most $m^3$ choices of $x_1,..., x_r$ for the given index $m$.
	Hence, there are at most $m^3$ conjugacy classes of closed subgroups $H$ in $G$ of index $m$, for which $\mathcal{S}(V,H)$ is boolean. Since each of these subgroups has at most $m$ conjugates, we obtain that $|\mathfrak{B}_m|\leq m^4$. 
\end{proof}

\subsection*{Proof in Case II}

First, we need to recall the notion of a flag in the vector space $V$. A \textbf{flag} $f$ in $V$ is a sequence $(0,W_1,\dots,W_k,V)$ of subspaces of $V$ such that $\,0<W_1<\dots<W_k<V$.
Clearly, $f$ can be regarded as a chain of subspaces from $0$ to $V$ in the subspace lattice of $V$. The set of all flags in $V$ is denoted by $Fl_V$, and $G=\GL(V)$ acts on $Fl_V$  in the following obvious way:
$$\forall\, g\in G \;\quad\; (0,W_1,\dots,W_k,V)^g = (0,W_1^g,\dots, W_k^g,V).$$
Let $f=(0,W_1,\dots,W_k,V)$ be a flag in $V$.
With some abuse of notation, the stabilizer in $G$ of $f$ is
$$\mathrm{stab}_G(f)=\{g\in G \mid W_i^g=W_i  \;\text{ for all }i=1,\dots,k\}=\bigcap_{i=1}^k \stab_G(W_i)\,,$$
where $\stab_G(W_i)$ denotes the stabilizer in $G$ of each subspace $W_i$ with respect to the usual action of $G$ on the subspace lattice of $V$. 
The stabilizer in $G$ of a flag in $V$ is also called a \textbf{parabolic subgroup} of $\mathrm{GL}(V)$ (see \cite[Chapter 3]{wilson09}).

\begin{oss}
	Looking at the closure operator $\cl$ of Definition \ref{defClosureOperatorGL/PGL},
	a subgroup $H$ of $G$ is parabolic if and only if $H$ is closed in $G$ and $\mathcal{S}(V,H)$ is a flag in $V$. 
	So, the number $|\mathfrak{F}_m|$ is exactly the number of parabolic subgroups of $\GL(V)$ of index $m$.
\end{oss} 

Finally, we say that $f=(0,W_1,\dots,W_k,V)$ is a \textbf{flag of type $(d_1,\dots,d_k)$} if $d_i=\dim(W_i)$ for each $i=1,\dots,k$.
We set 
$$Fl_{V}(d_1,\dots,d_k):=\{f \text{ flag in } V \mid f \text{ is of type } (d_1,\dots,d_k) \}$$
for every sequence of positive integers $d_1,\dots,d_k$ such that $0<d_1<\dots<d_k<n$. The group $G$ acts transitively on $Fl_V(d_1,\dots,d_k)$. Consequently, the stabilizers in $G$ of flags in $Fl_V(d_1,\dots,d_k)$ are conjugate to each other.

\begin{prop}\label{ClosedSbgrpsFlag}
	Let $V=\F_q^n$ and $G=\mathrm{GL}(V)$. 
	Then there exists an absolute constant $\alpha_2$, which is independent of $n$ and $q$, such that $|\mathfrak{F}_m|\leq m^{\alpha_2}\,$ for every $m\in\N$.
\end{prop}
\begin{proof}
	Let $m\in\N$ be fixed and consider $H\in\mathfrak{F}_m$.
	Then 
	$H$ is the stabilizer in $G$ of a flag $f$ in $V$ and $|G:H|=m$.  
	Let $f$ be the flag
	$0<W_1<\dots<W_k<V,$
	so that
	$$H=\mathrm{stab}_G(f)=\bigcap_{i=1}^k\mathrm{stab}_G(W_i).$$
	For each $i=1,\dots,k$, let $x_i$ be the dimension of the subspace $W_i\,$. 
	We consider the set $Fl_{V}(x_1,\dots,x_k)$ of flags of type $(x_1,\dots,x_k)$ in $V$. 
	Since the action of $G$ on $Fl_{V}(x_1,\dots,x_k)$ is transitive, we have that
	\begin{align*}
	|Fl_{V}(x_1,\dots,x_k)|= \frac{|\mathrm{GL}(V)|}{|\mathrm{stab}_G(f)|}=|G:H|\,.
	\end{align*}
	But we can also compute $|Fl_{V}(x_1,\dots,x_k)|$ as follows: 
	\begin{align*}
	|Fl_{V}(x_1,\dots,x_k)| & = |Fl_{V}(x_k)|\cdot|Fl_{W_k}(x_1,\dots,x_{k-1})| \\ \\
	& = |Fl_{V}(x_k)|\cdot|Fl_{W_k}(x_{k-1})|\cdot\ldots\cdot|Fl_{W_2}(x_{1})| \\ \\
	& = \binom{n}{x_k}_q \cdot \binom{x_k}{x_{k-1}}_q \cdot \ldots \cdot \binom{x_2}{x_{1}}_q\,.
	\end{align*}
	Then
	\begin{align*}
	m= |G:H| = |Fl_{V}(x_1,\dots,x_k)| = \binom{n}{x_k}_q \cdot \binom{x_k}{x_{k-1}}_q \cdot \ldots \cdot \binom{x_2}{x_{1}}_q\,.
	\end{align*}
	Let 
	$$v_k=\binom{n}{x_k}_q , \quad v_i=\binom{x_{i+1}}{x_i}_q\; \text{ for } i=1,\dots,k-1$$
	so that $m=v_k\cdot\ldots\cdot v_1$.
	As in the proof of Proposition \ref{ClosedSbgrpsBoolean}, we know that the number of such ordered factorizations of $m$ is at most $m^2$ by \cite{kalmar31}. Moreover, assume that the factorization $m=v_k\cdot\ldots\cdot v_1$ is fixed. Then, for each $i=1,\dots,k$ there are at most two possible values of $x_i$ for which we get $v_i$.
	So, for every fixed ordered factorization, we have at most $2^{k}$ possibilities, and $2^{k}\leq m$. Then there are at most $m^3$ choices of $x_1,...,x_k$ for the given index $m$.
	Hence, there are at most $m^3$ conjugacy classes of closed subgroups $H$ in $G$ of index $m$, for which $\mathcal{S}(V,H)$ is a flag. Each of these subgroups has at most $m$ conjugates, so $|\mathfrak{F}_m|\leq m^4$.
\end{proof}

\subsection*{Proof of Theorem \ref{ClosedSbgrpsProductOfChains}}

Here we combine Proposition \ref{ClosedSbgrpsBoolean} and Proposition \ref{ClosedSbgrpsFlag} in order to prove Theorem \ref{ClosedSbgrpsProductOfChains}. For the reader's convenience, we state the theorem again (for $G=\GL(V)$), using the notation we have introduced in this section. 
As remarked above, the analogous result for $\PGL(V)$ is an immediate consequence.

\begin{thm1.2*}
	Let $V=\F_q^n$ and $G=\mathrm{GL}(V)$. 
	Then there exists an absolute constant $\alpha$, independent of $n$ and $q$, such that $|\mathfrak{X}_m|\leq m^{\alpha}\,$ for every $m\in\N$.
\end{thm1.2*}
\begin{proof}
	Fix $m\in\N$ and consider $H\in\mathfrak{X}_m$. 
	So $H$ is closed in $G$, $|G:H|=m$, and 
	$\mathcal{S}(V,H)$ is isomorphic to the product of $r$ chains $\gamma_1,\dots,\gamma_r$. 
	For each $i=1,\dots,r$ there exists a vector subspace $W_{k_i}^{(i)}$ of $V$ such that the chain 
	$\gamma_i$ corresponds to a flag in $W_{k_i}^{(i)}$ of the form
	$$0=W_0^{(i)}<W_1^{(i)}<\dots<W_{k_i-1}^{(i)}<W_{k_i}^{(i)}\,.$$
	We identify this flag with the chain $\gamma_i$, so that $W_{k_i}^{(i)}$ represents the maximum of $\gamma_i$. In this way,
	every subspace $T\in\mathcal{S}(V,H)$ is identified with an $r$-tuple of subspaces $(W_{j_1}^{(1)},\dots,W_{j_r}^{(r)})\in\prod_{i=1}^r \gamma_i$, so that $j_i\in\{0,\dots,k_i\}$ and 
	$T=\bigoplus_{i=1}^r W_{j_i}^{(i)}\,.$
	In particular,
	\begin{equation}\label{eq_ProdCha1}
	V=\bigoplus_{i=1}^r W_{k_i}^{(i)}\,.
	\end{equation}
	Then every join-irreducible element of $\mathcal{S}(V,H)$ is one of the subspaces $W_j^{(i)}$, with $i\in\{1,\dots,r\}$ and $j\in\{1,\dots,k_i\}$. It is identified with an $r$-tuple of the form 
	$(0,\dots,0,W_{j}^{(i)},0,\dots,0)\in\prod_{s=1}^r \gamma_s$ such that $W_{j}^{(i)}\in\gamma_i$. 
	Since $H$ is closed in $G$, 
	the subgroup $H$ is uniquely determined by the join-irreducible elements of $\mathcal{S}(V,H)$, as follows:
	\begin{equation}
	H= \bigcap_{W\in \mathrm{JI}(\mathcal{S}(V,H))} \mathrm{stab}_G(W) = \bigcap_{i=1}^r\,\bigcap_{j=1}^{k_i}\,\mathrm{stab}_G(W_{j}^{(i)})\,.
	\end{equation} 
	By the proof of Proposition \ref{ClosedSbgrpsBoolean}, for each divisor $\overline{m}$ of $m$, we have at most $\overline{m}^{\,\alpha_1}$ (with $\alpha_1=4$) decompositions of $V$ as in \eqref{eq_ProdCha1} with
	\begin{equation*}
	\overline{m}=\frac{|\GL(V)|}{|\bigcap_{i=1}^r \stab_G(W_{k_i}^{(i)})|}\,.
	\end{equation*}
	The number of divisors of $m$ is clearly bounded by $m$.
	Thus, there are at most $m^5$ possible decompositions of $V$ as in \eqref{eq_ProdCha1} corresponding to a subgroup $H\in\mathfrak{X}_m$.
	Now choose such a decomposition
	$V=\bigoplus_{i=1}^r W_{k_i}^{(i)}$.
	For each $i=1,\dots,r$, we can consider the relative chain $C_i:=\gamma_i\cup\{V\}$, represented by a flag in $V$ of the form
	\begin{equation*}
	0=W_0^{(i)}<W_1^{(i)}<\dots<W_{k_i-1}^{(i)}<W_{k_i}^{(i)}<V\,.
	\end{equation*}
	Therefore,
	\begin{equation*}
	\frac{|\GL(V)|}{|\stab_G(C_i)|}=\binom{n}{x_{k_i}}_q \cdot \binom{x_{k_i}}{x_{k_i-1}}_q \cdot \ldots \cdot \binom{x_2}{x_{1}}_q\,
	\end{equation*}
	where $x_j:=\dim(W_j^{(i)})$ for each $j=1,\dots,k_i\,$.
	We observe that 
	$$\binom{n}{x_{k_i}}_q=\frac{|\GL(V)|}{|\stab_G(W_{k_i}^{(i)})|}$$
	has been fixed with the choice of the decomposition $V=\bigoplus_{i=1}^r W_{k_i}^{(i)}$. Then we set
	\begin{align}\label{eq_yi}
	y_i &:=\frac{|\GL(V)|}{|\stab_G(C_i)|}\cdot\frac{|\stab_G(W_{k_i}^{(i)})|}{|\GL(V)|}  =\binom{x_{k_i}}{x_{k_i-1}}_q \cdot \ldots \cdot \binom{x_2}{x_{1}}_q\,. 
	\end{align}
	As in the proof of Proposition \ref{ClosedSbgrpsFlag}, by \eqref{eq_yi} there are at most $y_i^{\alpha_2}$ (with $\alpha_2=4$) ways to choose the chain $C_i$, for all $i=1,\dots,r$. 
	We note that 
	\begin{equation}\label{eq_ProdCha2}
	y_i=\binom{x_{k_i}}{x_{k_i-1}}_q \cdot \ldots \cdot \binom{x_2}{x_{1}}_q=\frac{|\GL(W_{k_i}^{(i)})|}{|\stab_{\GL(W_{k_i}^{(i)})}(\gamma_i)|}
	\end{equation}
	and
	\begin{equation}\label{eq_ProdCha3}
	m=|G:H|=\frac{|\GL(V)|}{|\stab_{\GL(W_{k_1}^{(1)})}(\gamma_1)|\cdot\ldots\cdot|\stab_{\GL(W_{k_r}^{(r)})}(\gamma_r)|}\,.
	\end{equation}
	So, by equations \eqref{eq_ProdCha2} and \eqref{eq_ProdCha3}, we obtain that
	\begin{equation*}
	m\geq \frac{|\GL(W_{k_1}^{(1)})|}{|\stab_{\GL(W_{k_1}^{(1)})}(\gamma_1)|}\cdot\ldots\cdot\frac{|\GL(W_{k_r}^{(r)})|}{|\stab_{\GL(W_{k_r}^{(r)})}(\gamma_r)|}=y_1\cdot\ldots\cdot y_r\,.
	\end{equation*}
	Hence, $\,y_1^4\cdot\ldots\cdot y_r^4\leq m^4\,$ is a bound on the number of possible choices for the $r$-tuples $(\gamma_1,\dots,\gamma_r)$ of chains, once $y_1,\dots,y_r$ are fixed as above. But $y_1\cdot\ldots\cdot y_r$ is a factorization of a divisor $d$ of $m$. By \cite{kalmar31}, there are at most $d^2$ such factorizations for each divisor $d$. Thus, we can choose at most $m^3$ factorizations $y_1\cdot\ldots\cdot y_r$ corresponding to the fixed decomposition $V=\bigoplus_{i=1}^r W_{k_i}^{(i)}$. 
	We conclude that
	\begin{equation*}
	|\mathfrak{X}_m|\leq m^5\cdot m^4\cdot m^3 = m^{12}\,.
	\end{equation*}
\end{proof}

\subsection*{A final remark}

Let $m\in\N$ and let $z(m)$ be the number of closed subgroups in $\GL(V)$ of index $m$ that are the closure of subgroups generated by a cyclic endomorphism. 
Theorem \ref{ClosedSbgrpsProductOfChains} allows us to conclude that there exists an absolute constant $\alpha$ such that $z(m)$ is bounded by a polynomial function $m^\alpha\,$. 

This suggests studying the problem of determining a similar estimate for the number of closed subgroups in $\GL(V)$ containing some cyclic endomorphism of $V$.
In fact,
Neumann and Praeger showed that almost all elements of $\GL(V)$ are cyclic endomorphisms of $V$, as stated in the following theorem (they use the term \emph{cyclic matrix}, as the matrix group $\GL(n,q)$ is considered).

\begin{thm}[Theorem 3.1 in \cite{neuprae00}]\label{prop_CyclicMatrix-Proportion}
	Let $\mathrm{Cyc}(n,q)$ be the set of all cyclic matrices in $\GL(n,q)$. Let 
	$$P(\mathrm{Cyc}(n,q))=\frac{|\mathrm{Cyc}(n,q)|}{|\GL(n,q)|}$$
	be the probability that a matrix in $\GL(n,q)$ is cyclic. Then, for all $n\geq2$ and prime powers $q$
	$$1-P(\mathrm{Cyc}(n,q))\leq \frac{1}{q(q^2-1)}\,.$$
\end{thm}

So, we can think that many closed subgroups in $\GL(V)$ contain a cyclic endomorphism of $V$. For a subgroup $H$ containing a cyclic endomorphism $\xi$, a natural observation is that $\mathcal{S}(V,H)$ is a sublattice of $\mathcal{S}(V,\xi)$. Thus, the methods of the proof of Theorem \ref{ClosedSbgrpsProductOfChains} could be applied or extended to the case of sublattices of the product of chains $\mathcal{S}(V,\xi)$. 

This seems to be worthy of consideration: we could exploit again the interplay between subspaces and subgroups due to the closure operator on the subgroup lattice of $\GL(V)$, in order to give a more precise estimate on the number of closed subgroups in $\GL(V)$.

\bigskip
{\small 
	\textsc{Mathematisches Institut, Heinrich-Heine-Universit\"at  D\"usseldorf, \\  40225 D\"usseldorf, Germany}} 

{\small \textsl{E-mail address:} \texttt{luca.di.gravina@hhu.de}


\begin{thebibliography}{49}
	\addcontentsline{toc}{chapter}{Bibliography}
	{\small
				
		\bibitem{aschbacher84} M. Aschbacher, On the maximal subgroups of the finite classical groups, \emph{Inventiones Mathematicae}, \textbf{76} (1984), 469-514. 
		
		\bibitem{brifil67} L. Brickman and P. A. Fillmore, The Invariant Subspace Lattice of a Linear Transformation, \emph{Canadian Journal of Mathematics} \textbf{19} (1967), 810-822. 
		
		\bibitem{colombo-lucchini10} Valentina Colombo and Andrea Lucchini, On the subgroups with non-zero M\"obius numbers in the alternating and symmetric groups, \emph{Journal of Algebra}, \textbf{324} (2010), 2464-2474. 
		
		\bibitem{dallavolta-digravina22} F. Dalla Volta and L. Di Gravina, M\"obius function of the subgroup lattice of a finite group and Euler characteristic, \emph{ArXiv preprint}, 2022. \\ 	arXiv:2212.01917 [math.GR] 
		
		\bibitem{godsil18} C. D. Godsil, An Introduction to the Moebius Function, \emph{ArXiv preprint}, 2018. \\  arXiv:1803.06664v1 [math.CO] 
		
		\bibitem{hall36} P. Hall, The Eulerian functions of a group, \emph{Quarterly Journal of Mathematics} \textbf{7} (1936), 134–151. 

		\bibitem{hawkes-isaacs-oszaydin89} T. Hawkes, I. M. Isaacs and M. \"Ozaydin, On the M\"obius function of a finite group, \emph{Rocky Mountain Journal of Mathematics} \textbf{19} (1989), 1003-1034 
		
		\bibitem{kalmar31} Laszlo Kalmar, \"Uber die mittlere Anzahl der Produktdarstellungen der Zahlen. (Erste Mitteilung), \emph{Acta Litterarum ac Scientiarum, Szeged} \textbf{5} (1931), 95-107. 
		\bibitem{kleidlieb90} Peter Kleidman and Martin Liebeck, \emph{The Subgroup Structure of the Finite Classical Groups}, Cambridge University Press, Cambridge, 1990. 
		
		\bibitem{kratzer-thevenaz84} Charles Kratzer and Jacques Thévenaz, Fonction de M\"obius d'un groupe fini et anneau de Burnside, \emph{Comment. Math. Helvetici} \textbf{59} (1984), 425-438. 
		
		\bibitem{lucchini10} Andrea Lucchini, On the subgroups with non-trivial M\"obius number, \emph{Journal of Group Theory}, \textbf{13} (2010), 589-600. 
		
		\bibitem{mann05} Avinoam Mann, A probabilistic zeta function for arithmetic groups, \emph{International Journal of Algebra and Computation}, \textbf{15} (2005), 1053-1059. 
		
		\bibitem{neuprae00} Peter M. Neumann and Cheryl E. Praeger, Cyclic Matrices in Classical Groups over Finite Fields, \emph{Journal of Algebra} \textbf{234} (2000), 367–418. 
		
		\bibitem{stanley86} Richard P. Stanley, \emph{Enumerative Combinatorics, Volume I}, second ed., Cambridge University Press, Cambridge, 2012 (first ed. published by Wadsworth \& Brooks/Cole, 1986). 
		
		\bibitem{wilson09} Robert A. Wilson, \emph{The Finite Simple Groups}, Springer-Verlag, London, 2009. 
	}
\end{thebibliography}
\end{document}